\documentclass[12pt]{amsart}
\title[]{Centralizers in $\widetilde A_2$ groups}
\author{Guyan Robertson}
\address{School of Mathematics and Statistics, Newcastle University, Newcastle upon Tyne NE1 7RU, England, UK}
\email{a.g.robertson@ncl.ac.uk}

\date{\today}

\usepackage{amsmath}
\usepackage{amscd}
\usepackage{amssymb}

\input{prepictex}
\input{pictex}
\input{postpictex}
\chardef\bslash=`\\ 

\makeatletter
\def\verbatim{\interlinepenalty\@M \@verbatim
  \leftskip\@totalleftmargin\advance\leftskip2pc
  \frenchspacing\@vobeyspaces \@xverbatim}
\makeatother
\newtheorem{theorem}{Theorem}[section]

\newtheorem{lemma}[theorem]{Lemma}

\theoremstyle{definition}

\newtheorem{definition}[theorem]{Definition}
\newtheorem{remark}[theorem]{Remark}

\newcommand{\cl}[1]{{\mathcal{#1}}}
\newcommand{\bb}[1]{{\mathbb{#1}}}

\newcounter{picture}

\DeclareMathOperator{\conv}{conv}
\DeclareMathOperator{\diam}{diam}
\DeclareMathOperator{\dist}{dist}
\DeclareMathOperator{\Z}{Z}

\newcommand{\Min}{{\text{\rm Min}}}
\newcommand{\aut}{{\text{\rm Aut}}}
\newcommand{\PGL}{{\text{\rm{PGL}}}}

\begin{document}

\begin{abstract} Let $\Gamma$ be a torsion free discrete group acting cocompactly on a two dimensional euclidean building $\Delta$.
The centralizer of an element of $\Gamma$ is either a Bieberbach group or is described by a finite graph of finite cyclic groups.
Explicit examples are computed, with $\Delta$ of type $\widetilde A_2$.
\end{abstract}

\maketitle

\section{Introduction}

Let $\Gamma$ be a torsion free discrete group which acts
cocompactly on a 2-dimensional euclidean building $\Delta$
\cite{br, ga, ron}. For example, $\Gamma$ may be any torsion free
lattice in $\PGL_3(\bb F)$, where $\bb F$ is a nonarchimedean
local field. In that case $\Gamma$ acts cocompactly on the
Bruhat-Tits building $\Delta$ of $\PGL_3(\bb F)$. The building
$\Delta$ is considered as a geodesic metric space \cite[Chapter
VI.3]{br} and is a union of apartments which are isometric to the
euclidean plane. Define the \textit{translation length} of an
element $g\in \Gamma$ to be $|g|= \inf_{x\in \Delta} d(g.x, x)$,
and let
$$\Min(g)=\{x\in \Delta : d(g.x, x) = |g|\}.$$
An element $g\in \Gamma$ is \textit{hyperbolic} if $|g|>0$ and
$\Min(g)$ is nonempty.  Since $\Gamma$ is torsion free, all nontrivial elements of $\Gamma$ are hyperbolic
 and $\Gamma$ acts freely on $\Delta$.
Each element $g\in \Gamma - \{1\}$ has at least one geodesic line $\ell$ in $\Delta$ (called an \textit{axis} for $g$) upon which $g$ acts by translation.
 The axes of $g$ are parallel to each other and their union is the convex set $\Min(g)$ \cite[Theorem II.6.8]{bh}.
 An axis $\ell$ for $g$ will always be \textit{oriented} in the direction of translation by $g$, which will be called the positive direction.
A nontrivial element $g\in \Gamma$ is said to be \textit{regular} if no axis of $g$ is parallel to a wall of $\Delta$;
 otherwise $g$ is said to be \textit{irregular}.
  If $g$ is regular then every axis of $g$ is contained in a unique apartment, as illustrated in Figure \ref{regular}.

 \refstepcounter{picture}
\begin{figure}[htbp]
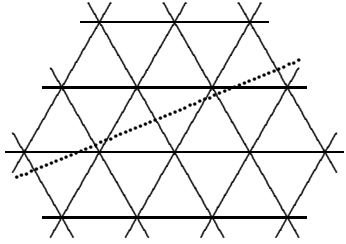
\label{regular}
\centerline{
\beginpicture
\setcoordinatesystem units  <0.5cm, 0.866cm>        
\setplotarea x from -3.5 to 5, y from -1.5 to 2.5   
\putrule from -2.5   2     to  2.5  2
\putrule from  -3.5 1  to 3.5 1
\putrule from -4.5 0  to 4.5 0
\putrule from -3.5 -1  to  3.5 -1
\setlinear
\plot  -4.3 -0.3  -1.7 2.3 /
\plot -3.3 -1.3  0.3 2.3 /
\plot -1.3 -1.3  2.3 2.3 /
\plot  0.7 -1.3  3.3 1.3 /
\plot  2.7 -1.3  4.3 0.3 /
\plot  -4.3 0.3  -2.7 -1.3 /
\plot  -3.3 1.3   -0.7 -1.3 /
\plot  -2.3 2.3   1.3 -1.3 /
\plot  -0.3 2.3   3.3 -1.3 /
\plot  1.7 2.3   4.3 -0.3 /
\setplotsymbol({$\cdot$}) \plotsymbolspacing=2pt
\plot  -4.2 -0.4   3.3 1.4 /
\endpicture
}
\caption{Axis of a regular element}
\end{figure}

The main result of this article concerns the structure of centralizers.

\begin{theorem}\label{main}
With the above hypotheses, let $g\in\Gamma - \{1\}$ and let $\Z_\Gamma(g)$ denote the centralizer of $g$ in $\Gamma$.
\begin{itemize}
\item[(a)]
 If $g$ has only one axis then $\Z_\Gamma(g) \cong \bb Z$.
\item[(b)]
  If $g$ has more than one axis then one of the following is true.
\begin{itemize}
\item[$\bullet$]
If $g$ is regular then $\Z_\Gamma(g)$ is a Bieberbach group in two dimensions.
\item[$\bullet$]
If $g$ is irregular then $\Z_\Gamma(g)/\langle g \rangle$ is the fundamental group of a finite graph of finite cyclic groups.
\end{itemize}
\end{itemize}
\end{theorem}

\noindent The principal aim of this article is to show how to compute centralizers explicitly.
 This is done in Section \ref{examples} below in the $\widetilde A_2$ case.
General results on the structure of centralizers in semihyperbolic groups are to be found in \cite[Section 7]{ab}, \cite[III.4]{bh}.

In the non-exotic case, where the building is associated to a
linear algebraic group of which $\Gamma$ is a subgroup, the current
results can be deduced and interpreted using matrices.  All locally
finite euclidean buildings of dimension $\geq 3$ are associated to linear
algebraic groups, and, up to a point, this justifies the restriction
to two-dimensional buildings in the present article.

\section{Background}\label{background}

Since $\Gamma\backslash\Delta$ is compact, $\Delta$ is uniformly locally finite, in the sense that there exists $q\ge 2$ such that
any edge of $\Delta$ lies on at most $q+1$ triangles. If $\Delta$ is of type $\widetilde A_2$ then every edge of $\Delta$ lies on exactly $q+1$ triangles.
This follows from the facts that the link of a vertex in a building of type $\widetilde A_2$
is a generalized 3-gon \cite[Chapter 3.2]{ron}, and all vertices in a thick
generalized 3-gon have the same valency \cite[Proposition (3.3)]{ron}.
The figures below relate to buildings of type $\widetilde A_2$,
in which apartments are euclidean planes tessellated by equilateral triangles.

The main result is based on the observation that if $g, h\in \aut(\Delta)$ commute, then $h$ leaves $\Min(g)$ invariant and $h$ maps each axis of $g$ to an axis of $g$ \cite[Theorem II.6.8]{bh}.

 \begin{lemma}\label{b}
  If an element $g\in \Gamma-\{1\}$ has only one axis then $\Z_\Gamma(g) \cong \bb Z$.
\end{lemma}

\begin{proof}
If $g$ has only one axis $\ell$ then $\Min(g)=\ell$ and so $\Z_\Gamma(g)$ acts freely by translation on $\ell$. Therefore $\Z_\Gamma(g)=\bb Z$.
\end{proof}

This proves Theorem \ref{main}(a). It is now only necessary to consider elements with more than one axis. The next result proves the first part of
Theorem \ref{main}(b).

 \begin{lemma}
  Suppose that $g\in \Gamma-\{1\}$ is a regular element. Then $\Z_\Gamma(g)$ is a Bieberbach group, isomorphic to either $\bb Z^2$
  or the fundamental group of a Klein bottle.
\end{lemma}

\begin{proof}
   Let $\ell$ be an axis for $g$.  Then $\ell$ is not parallel to any wall and lies in a unique apartment $\cl A$ of $\Delta$,
   which is $g$-invariant, since $\ell$ is.
   Any other axis $\ell^\prime$ for $g$  is parallel to $\ell$ and therefore lies in a common apartment with $\ell$.
   However $\ell^\prime$ also lies in a unique apartment, namely $\cl A$.
   Thus $\Min(g)\subseteq\cl A$. Now since $g$ has more than one axis,
   $g$ acts by translation (rather than glide reflection) on $\cl A$,
   and any geodesic in $\cl A$ which is parallel to $\ell$ is an axis for $g$.
   Thus $\Min(g)= \cl A$ and $\Z_\Gamma(g)$ leaves $\cl A$ invariant.

   We claim that $\Z_\Gamma(g)$ acts cocompactly on $\cl A$.
   Choose a  point $\xi$ on $\ell$ and a geodesic $p \not=\ell$ in $\cl A$ passing through $\xi$.
   Let $D\subset\Delta$ be a compact fundamental domain for $\Gamma$ containing $\xi$, and choose $\xi_n\in p,\, n=1,2,\dots$ such
   that $\dist(\xi_n, \xi) \ge 2n.\diam(D)$. Let $\ell_n$ be the geodesic in $\cl A$ passing through $\xi_n$ and parallel to $\ell$.
   Then $\ell_n$ is an axis for $g$, since $g$ acts by translation on $\cl A$ in the direction of $\ell$.
    Choose $g_n\in \Gamma$ such that $g_n\xi_n\in D$. The geodesic $g_n\ell_n$ passes through $D$
   and is an axis for $g_ngg_n^{-1}$. The displacement of $g_n\xi_n$ by $g_ngg_n^{-1}$ is equal to the displacement of $\xi$ by $g$
   and each $g_n\xi_n$ lies in $D$. Since $\Gamma$ is discrete, there are infinitely many pairs $m\not= n$ such that
   $g_mgg_m^{-1} = g_ngg_n^{-1}$. Choose such a pair. Then $h=g_m^{-1}g_n\in \Z_\Gamma(g)$ and so $h$ leaves $\cl A$ invariant.
   By considering $h^2$ if necessary, we may assume that $h$ is a translation rather than a glide reflection of $\cl A$
   and $h$ moves $\xi$ away from $D$. Thus the subgroup of $\Z_\Gamma(g)$ generated by
   $g$ and $h$ acts cocompactly on $\cl A$.
   It follows that $\Z_\Gamma(g)$ is a Bieberbach group in two dimensions.
\end{proof}

\begin{definition}
  A \textit{strip} in $\Delta$ is the convex hull of two parallel walls of $\Delta$ which contains no other wall of $\Delta$
  (and so is of minimal width). Recall that a wall in $\Delta$ is a geodesic in the 1-skeleton of an apartment which is fixed
  pointwise by a reflection in the associated euclidean Coxeter group. If $g\in \Gamma - \{1\}$ then an \textit{axial wall}
  for $g$ is a wall of $\Delta$ which is also an axis for $g$.
\end{definition}

\begin{lemma}\label{a}
  Let $g\in \Gamma-\{1\}$ be an irregular element, with more than one axis.
Then any axis $\ell$ for $g$ is either an axial wall or lies in the interior of a unique strip bounded by
two axial walls.
\end{lemma}

\begin{proof}
Suppose that $\ell$ is an axis  for $g$ which is not a wall of $\Delta$.
Let $\Sigma$ be the unique strip containing $\ell$.
Since $g$ has more than one axis and $\Min(g)$ is convex,
$\Sigma$ must contain an axis $\ell^\prime$ for $g$ which is neither a wall nor the median line of $\Sigma$.
Let $\ell_0$ be the unique wall parallel to $\ell^\prime$ at minimal distance $\delta$ to $\ell^\prime$.
Since $g$ is an isometry, it must map $\ell_0$ to a parallel wall $g.\ell_0$, also at distance $\delta$ to $\ell^\prime$.
By uniqueness, $\ell_0=g.\ell_0$, and $\ell_0$ is an axis of $g$.
Finally, the edge $\ell_1$ of $\Sigma$ opposite to $\ell_0$ is also an axis of $g$.
\end{proof}

  The preceding argument shows that, if an irregular element $g\in\Gamma$ has only one axis $\ell$, then $\ell$ is
  either a median line or a wall.

\begin{definition}
\label{tree}
Suppose that the element $g\in \Gamma -\{1\}$ is irregular with more than one axis. Define a graph $\cl T_g$ as follows. A vertex $\ell$ of $\cl T_g$ is an axial wall of $g$.
There is an edge $[\ell_1, \ell_2]$ between two such vertices if $\ell_1$, $\ell_2$ are the boundary
walls of a common strip.
\end{definition}

\begin{lemma}\label{d}
  Under the above hypotheses, the graph $\cl T_g$ is a tree.
\end{lemma}

\begin{proof}
  If $\ell$, $\ell'$ are axial walls of $g$, then their convex hull $\conv(\ell,\ell')$  is a convex subset of an apartment of $\Delta$,
   which is the union of $n$ contiguous parallel strips. This defines a path of length $n$ in $\cl T_g$ from $\ell$ to $\ell'$.
    The path is unique, since $\Delta$ is two dimensional and contractible \cite [Appendix 4, page 185]{ron}.
\end{proof}

\refstepcounter{picture}
\begin{figure}[htbp]
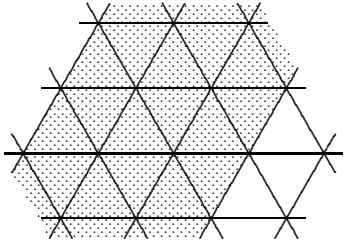

\centerline{
\beginpicture
\setcoordinatesystem units  <0.5cm, 0.866cm>        
\setplotarea x from -3.5 to 5, y from -1.5 to 2   
\putrule from -2.5   2     to  2.5  2
\putrule from  -3.5 1  to 3.5 1
\putrule from -4.5 0  to 4.5 0
\putrule from -3.5 -1  to  3.5 -1
\setlinear
\plot  -4.3 -0.3  -1.7 2.3 /
\plot -3.3 -1.3  0.3 2.3 /
\plot -1.3 -1.3  2.3 2.3 /
\plot  0.7 -1.3  3.3 1.3 /
\plot  2.7 -1.3  4.3 0.3 /
\plot  -4.3 0.3  -2.7 -1.3 /
\plot  -3.3 1.3   -0.7 -1.3 /
\plot  -2.3 2.3   1.3 -1.3 /
\plot  -0.3 2.3   3.3 -1.3 /
\plot  1.7 2.3   4.3 -0.3 /
\setshadegrid span <1.5pt>
\vshade  -4.3  -0.3  -0.3  <,,,z>  -3.3  -1.3  0.7  <z,,,>  -1.7   -1.3  2.3  <z,,,>  0.7  -1.3 2.3 <z,,,>  2.3  0.3 2.3 <z,,,>  3.3   1.3 1.3 /
\endpicture
}
\caption{Part of $\Min(g)$ corresponding to path of length 3 in $\cl T_g$}
\end{figure}

  If $\ell$ is any wall of $\Delta$ then the class of all walls parallel to $\ell$ is a wall $\ell^\infty$ of the
  spherical building at infinity $\Delta^\infty$. There is an associated tree $T(\ell^\infty)$,
  whose vertices are the elements $\ell^\infty$ \cite[Chapter 10.2]{ron}, \cite[Chapter 10]{weiss}. If $\ell$ is an axis of the element $g\in \Gamma$ then the tree $\cl T_g$ is a subtree of $T(\ell^\infty)$.

\begin{lemma}\label{e}
  The group $G=\Z_\Gamma(g)/\langle g \rangle$ acts on the tree $\cl T_g$.
\end{lemma}

\begin{proof}
Since each element of $\Z_\Gamma(g)$ maps axes of $g$ to axes of $g$, the group $\Z_\Gamma(g)$ acts on $\cl T_g$.
Also the cyclic group $\langle g \rangle$ fixes each vertex of $\cl T_g$, since each axis is $g$-invariant.
\end{proof}

\begin{remark}\label{e*}
The action of $G$ on $\cl T_g$ may have inversion.
In order to obtain an action without inversion, each geometric edge which is stabilized by a nontrivial element of $G$ is divided into two
by inserting a vertex $\ell^*$ at its midpoint. Then $G$ acts without inversion of the resulting tree $\cl T_g^*$.
The full barycentric subdivision of $\cl T_g$ would be a good alternative to $\cl T_g^*$, but we have chosen to minimize the number of additional
vertices in order to simplify the graphs for the explicit examples below.
\end{remark}

\noindent
\begin{minipage}{8.5cm}

\begin{remark}\label{valency} Any strip containing an axial wall $\ell$ is completely determined by $\ell$ and a single triangle in the strip which contains
 a fixed edge of $\ell$. Since there are at most $q+1$ such triangles, each vertex of the tree $\cl T_g$ has valency at most $q+1$.
\end{remark}

\begin{remark}\label{start1}
Recall that an axis $\ell$ for $g$ is oriented in the direction of translation by $g$.
  Let $\ell, \ell^\prime$ be neighbouring axial walls for $g$ and fix an edge $\varepsilon = [v_0, w_0]$ in $\Delta$ joining them.
  Then $g$ translates $v_0 ,w_0$ the same distance along the parallel walls $\ell, \ell^\prime$.
  The edge $g\varepsilon = [gv_0, gw_0]$, is the edge obtained by translating $\varepsilon$ along the strip.
  This observation will be crucial in constructing the explicit examples in Section \ref{examples}.
\end{remark}
\end{minipage}
\hspace{1cm}
\begin{minipage}{3cm}
\beginpicture
\setcoordinatesystem units  <1.4cm, .8cm>  
\setplotarea  x from -1.2 to 0.1,  y from -1 to 1        
\put {$\varepsilon$} at  -0.4 -0.8
\arrow <8pt> [.2, .67] from    -0.5 -0.5 to -0.7 -0.7
\put {$g\varepsilon$}  at  -0.5 3.9
\arrow <8pt> [.2, .67] from    -0.5 3.5 to  -0.7 3.3
\setlinear \plot -1 -1  0 0  -1 1  0 2  -1 3  0 4 / 
\putrule from  -1 -1 to  -1 5
\putrule from  0 -1 to  0 5
\put {$_\bullet$} at -1 -1
\put {$v_0$} at -1.2 -1
\put {$_\bullet$} at -1  3
\put {$gv_0$}      at -1.3  3
\put {$\ell$}  at  -1.2 5
\put {$\ell^\prime$}  at  0.2 5
\put {$_\bullet$} at 0 0
\put {$w_0$} at 0.2 0
\put {$_\bullet$} at 0  4
\put {$gw_0$}      at 0.3  4
\setshadegrid span <1.5pt>
\vshade  -1  -1  3  <,z,,>   0   0  4  /
\endpicture
\end{minipage}

\bigskip

\noindent Recall that our axes are always \emph{oriented}.
\begin{lemma}\label{Z-orbit}
  Two axes $\ell, \ell^\prime$ lie in the same $\Z_\Gamma(g)$-orbit if and only if they lie in the same $\Gamma$-orbit.
\end{lemma}

\begin{proof}
  Recall that $g$ acts by translation by the same distance on each of the axes.
  In order to prove the nontrivial implication, let $\gamma\in \Gamma$ with $\gamma\ell=\ell^\prime$.
  We claim that $\gamma\in\Z_\Gamma(g)$. If $v\in \ell$ then $\gamma v\in \ell^\prime$ and
  $\gamma[v,gv]=[\gamma v, g\gamma v]$. Therefore $\gamma(gv)=g(\gamma v)$. Since $\Gamma$ acts freely on $\Delta$,
  $\gamma g =g\gamma$.
\end{proof}

\begin{lemma}\label{vertex orbit}
  Let $g\in \Gamma-\{1\}$. Then the set of $\Z_\Gamma(g)$-orbits of axial walls for $g$ is finite.
\end{lemma}

\begin{proof}
Since $g$ acts on each axis by translation of distance $|g|$, the
axis is determined by a segment of length $|g|$. Since $\Gamma$
acts cocompactly, every segment of length~$|g|$ can be translated
by some element $h\in \Gamma$ to some segment belonging to a
certain finite set.  Each segment in the finite set generates an
axial wall for some conjugate of~$g$.  So, further translating by
one of a finite number of elements $k\in \Gamma$, we obtain one of
a finite number of segments generating axial walls for~$g$ itself.
According to Lemma \ref{Z-orbit}, the elements $hk$ lie
in~$Z_\Gamma(g)$.
\end{proof}

\begin{lemma}\label{vertex stabilizer}
  Let $g\in \Gamma-\{1\}$ and let $\ell$ be an axis for $g$.
  Then the stabilizer $S(\ell)$ of $\ell$ in $\Z_\Gamma(g)$ is a cyclic group.
  If $S(\ell)$ is nontrivial then it is generated by an element of minimal positive translation length.
\end{lemma}

\begin{proof}
Each element $z$ of $S(\ell)$ acts upon the oriented axis $\ell$ by a translation of signed distance
$\rho(z) = \pm |z|$.
The map $z\mapsto \rho(z)$ is a homomorphism from $S(\ell)$ into $\bb R$, which is injective since $\Gamma$ acts freely on $\Delta$.
The set $\{|x| : x\in\Gamma\}$ of translation lengths of elements of $\Gamma$ is discrete.
Therefore the map $z\mapsto \rho(z)$ is an isomorphism from $S(\ell)$ onto a discrete subgroup of $\bb R$. The result follows.
\end{proof}

\section{The graph of groups and proof of the main theorem}\label{Bass-Serre}

The results of the previous section show that $G=\Z_\Gamma(g)/\langle g \rangle$ acts without inversion on the tree $\cl T = \cl T_g^*$.  Recall the construction of the corresponding graph of groups.
Since $G$ acts without inversion, we can choose an orientation on the edges which is invariant under $G$ \cite[I.3.1]{ser}. This orientation consists of a partition of the set of edges and a bijective involution $e\mapsto \overline e$ which interchanges the two components. Each directed edge $e$ has an initial vertex $o(e)$ and a terminal vertex $t(e)$ such that $o(\overline e)=t(e)$.
The quotient graph $X=G\backslash\cl T$ has vertex set $V$ and directed edge set $E$.  There is an induced involution $e\mapsto \overline e : E \to E$ and there are maps $o,t : E \to V$ with $o(\overline e)=t(e)$.

 Choose a maximal tree $T$ in $X$, and a lifting $\sigma : T \to \cl T$ to a subtree of $\cl T$. For each $v\in V$, let $G_v$ be the stabilizer of $\sigma(v)$ in $G$. If $e$ is an edge of $T$, define  $G_e=G_{\overline e}$ to be the stabilizer of $\sigma(e)$ in $G$ and define $\varphi_e : G_e \to G_{t(e)}$ to be the inclusion map. For each $e\in E_+- T$, define $\sigma(e)$ to be the directed edge of $\cl T$ lying over $e$ and with $o(\sigma(e))=\sigma(o(e))$, and define $G_e$ to be the stabilizer of $\sigma(e)$. We take $\varphi_{\overline e}$ to be the inclusion map, but as $\sigma(\overline e)\not=\overline{\sigma(e)}$, we choose $\gamma_e\in G$ with $\gamma_e (\sigma(\overline{e}))=\overline{\sigma(e)}$. Then $\varphi_e$ is defined as conjugation by $\gamma_e$.

Now $(X, G)$ is a graph of groups in the sense of \cite[I.5]{ser}.
By Bass-Serre theory \cite[I.5.4]{ser}, $G$ is isomorphic to the fundamental group $\pi_1(X, G)$, which is generated by the groups $G_v$ and the elements $\gamma_e$ subject to the following relations, for $e\in E$.

\begin{itemize}
  \item $\gamma_{\overline e}=\gamma_e^{-1}$;
  \item $\gamma_e=1, \quad e\in T$;
  \item $\gamma_e\varphi_e(x)\gamma_e^{-1}=\varphi_{\overline e}(x), \quad x\in G_e$.
\end{itemize}

The next result, which is an immediate consequence of the definitions, is useful in computing the fundamental group in the examples of the next section.

\begin{lemma}\label{edgemap}
  If $\varphi_e$ is surjective then $\gamma_eG_{t(e)}\gamma_e^{-1}\subseteq G_{o(e)}$.
  In particular, if $e\in T$, then $G_{t(e)}\subseteq G_{o(e)}$.
\end{lemma}

\bigskip

\noindent\textit{Proof of Theorem \ref{main}}.
This follows from the results of the previous section.
 The graph $X$ has finitely many vertices by Lemma \ref{vertex orbit},
and finitely many edges by Remark \ref{valency}. The stabilizer in $G$ of any vertex of $\cl T$ is a finite cyclic group by
Lemma \ref{vertex stabilizer}. The same is true of edge stabilizers, since they are subgroups of vertex stabilizers.
\qed

\section{Examples in an $\widetilde A_2$ group}\label{examples}

A class of groups which act regularly on the vertices of a building of type $\widetilde A_2$ has been studied in \cite{cmsz}.
Consider the group C.1 of \cite{cmsz}. This group was originally defined in \cite{mum}: see also \cite[Theorem 2.2.1]{ko}, \cite{gm}.
It is a torsion free lattice in $\PGL_3({\bb Q}_2)$, with seven generators
$x_0, x_1, \dots, x_6$
and seven relators
$$
x_0x_0x_6\quad
x_0x_2x_3\quad
x_1x_2x_6\quad
x_1x_3x_5\quad
x_1x_5x_4\quad
x_2x_4x_5\quad
x_3x_4x_6
$$
The $1$-skeleton of  the Bruhat-Tits building $\Delta$ of $\PGL_3({\bb Q}_2)$ is the Cayley graph of $\Gamma$ relative to this set of generators.
Vertices in $\Delta$ are identified with elements of $\Gamma$ and the directed edge $[a,ax]$, where $a\in\Gamma$, is
labelled by the generator $x$. To facilitate computations, it is convenient to tabulate,
for each $i$, the values of $j$ for which there is a relation of the form $x_ix_jx_k=1$.
\bigskip

\centerline{
\begin{tabular}{|c|c|c|c|c|c|c|c|}
\hline
$i$  &  0 & 1 & 2 & 3 & 4 & 5 & 6  \\
$j$  &  0, 2, 6 & 2, 3, 5 & 3, 4, 6 & 0, 4, 5 & 1, 5, 6 &
1, 2, 4 & 0, 1, 3   \\\hline
\end{tabular}
}
\bigskip

\subsection{Example: the tree $\cl T_g^*$ for $g=x_0x_5$}
The element $g$ has an axis $\ell$ passing through $1, g, g^2, \dots$, and $g$ acts on $\ell$ by translation of distance $|g|=2$.
We seek all possible axial walls for $g$ which are neighbours of $\ell$.
By Remark \ref{start1}, any strip containing $\ell$ and another axial wall must have periodic labelling of period 2.
The entire strip is completely determined by $\ell$ and a single triangle containing a fixed edge of $\ell$.
Therefore there are at most three such strips.
In fact, there are exactly three, corresponding to three edges in $\cl T_g$ joining $\ell$
to vertices $\ell_1, \ell_2, \ell_3$.
The three cases are illustrated on the left of Figure \ref{3neighbours}, where the edge label $j$ stands for the generator $x_j$.
The stabilizer in $\Z_\Gamma(g)$ of each strip is the subgroup $\langle g \rangle$, and this acts upon the strip
by translation. In each case, a fundamental domain for this action is shaded.

\refstepcounter{picture}
\begin{figure}[htbp]\label{3neighbours}
\hfil
\centerline{
\beginpicture
\setcoordinatesystem units  <1.4cm, .8cm>
\setplotarea  x from -1.2 to 0.1,  y from -1 to 5        
\arrow <8pt> [.2, .67] from   -1 0 to   -1 0.2
\arrow <8pt> [.2, .67] from    -0.5 -0.5 to -0.7 -0.7
\arrow <8pt> [.2, .67] from   -0.5 0.5 to -0.3 0.3
\arrow <8pt> [.2, .67] from    0 0.8 to  0 1
\arrow <8pt> [.2, .67] from    0 2.8 to  0 3
\putrule from -1 -1 to -1 1
\setlinear \plot -1 -1 0 0 -1 1 /
\put {$_{0}$}  at  -1.2 0
\put {$_{6}$} at  -0.4 -0.7
\put {$_{0}$}  at  -0.6 0.3
\put {$_{2}$}  at   0.2 1
\put {$_{2}$}  at   0.2 3
\arrow <8pt> [.2, .67] from   -1 2  to   -1 2.2
\arrow <8pt> [.2, .67] from    -0.5 1.5 to  -0.7 1.3
\arrow <8pt> [.2, .67] from    -0.5 2.5 to  -0.3 2.3
\put {$_{3}$}  at  -0.4 1.3
\put {$_{1}$}  at  -0.5 2.8
\put {$_{5}$}  at  -1.2 2
\putrule from  -1 1 to  -1 3
\putrule from 0 -1 to  0 5  
\setlinear \plot -1 1  0 2  -1 3 /
\arrow <8pt> [.2, .67] from   -1 4  to   -1 4.2
\arrow <8pt> [.2, .67] from    -0.5 3.5 to  -0.7 3.3
\arrow <8pt> [.2, .67] from    -0.5 4.5 to  -0.3 4.3
\put {$_{6}$}  at  -0.5 3.8
\put {$_{0}$}  at  -0.5 4.8
\put {$_{0}$}  at  -1.2 4
\putrule from  -1 3 to  -1 5
\setlinear \plot  -1 3  0 4  -1 5 /
\put {$\ell$}  at  -1 5.5
\put {$\ell_1$}  at  0 5.5
\setshadegrid span <1.5pt>
\vshade  -1  -1  3  <,z,,>   0   0 4  /   
\setcoordinatesystem units  <1.4cm, .8cm>  point at -2 0
\setplotarea  x from -1.2 to 0.1,  y from -1 to 1        
\arrow <8pt> [.2, .67] from   -1 0 to   -1 0.2
\arrow <8pt> [.2, .67] from    -0.5 -0.5 to -0.7 -0.7
\arrow <8pt> [.2, .67] from   -0.5 0.5 to -0.3 0.3
\arrow <8pt> [.2, .67] from    0 0.8 to  0 1
\arrow <8pt> [.2, .67] from    0 2.8 to  0 3
\putrule from -1 -1 to -1 1
\setlinear \plot -1 -1 0 0 -1 1 /
\put {$_{0}$}  at  -1.2 0
\put {$_{0}$} at  -0.4 -0.7
\put {$_{6}$}  at  -0.6 0.3
\put {$_{3}$}  at   0.2 1
\put {$_{3}$}  at   0.2 3
\arrow <8pt> [.2, .67] from   -1 2  to   -1 2.2
\arrow <8pt> [.2, .67] from    -0.5 1.5 to  -0.7 1.3
\arrow <8pt> [.2, .67] from    -0.5 2.5 to  -0.3 2.3
\put {$_{4}$}  at  -0.4 1.3
\put {$_{2}$}  at  -0.5 2.8
\put {$_{5}$}  at  -1.2 2
\putrule from  -1 1 to  -1 3
\putrule from 0 -1 to  0 5  
\setlinear \plot -1 1  0 2  -1 3 /
\arrow <8pt> [.2, .67] from   -1 4  to   -1 4.2
\arrow <8pt> [.2, .67] from    -0.5 3.5 to  -0.7 3.3
\arrow <8pt> [.2, .67] from    -0.5 4.5 to  -0.3 4.3
\put {$_{0}$}  at  -0.5 3.8
\put {$_{6}$}  at  -0.5 4.8
\put {$_{0}$}  at  -1.2 4
\putrule from  -1 3 to  -1 5
\setlinear \plot  -1 3  0 4  -1 5 /
\put {$\ell$}  at  -1 5.5
\put {$\ell_2$}  at  0 5.5
\setshadegrid span <1.5pt>
\vshade  -1  -1  3  <,z,,>   0   0 4  /   
\setcoordinatesystem units  <1.4cm, .8cm>  point at -4 0
\setplotarea  x from -1.2 to 0.1,  y from -1 to 5        
\arrow <8pt> [.2, .67] from   -1 0 to   -1 0.2
\arrow <8pt> [.2, .67] from    -0.5 -0.5 to -0.7 -0.7
\arrow <8pt> [.2, .67] from   -0.5 0.5 to -0.3 0.3
\arrow <8pt> [.2, .67] from    0 0.8 to  0 1
\arrow <8pt> [.2, .67] from    0 2.8 to  0 3
\putrule from -1 -1 to -1 1
\setlinear \plot -1 -1 0 0 -1 1 /
\put {$_{0}$}  at  -1.2 0
\put {$_{3}$} at  -0.4 -0.7
\put {$_{2}$}  at  -0.6 0.3
\put {$_{6}$}  at   0.2 1
\put {$_{6}$}  at   0.2 3
\arrow <8pt> [.2, .67] from   -1 2  to   -1 2.2
\arrow <8pt> [.2, .67] from    -0.5 1.5 to  -0.7 1.3
\arrow <8pt> [.2, .67] from    -0.5 2.5 to  -0.3 2.3
\put {$_{1}$}  at  -0.4 1.3
\put {$_{4}$}  at  -0.5 2.8
\put {$_{5}$}  at  -1.2 2
\putrule from  -1 1 to  -1 3
\putrule from 0 -1 to  0 5  
\setlinear \plot -1 1  0 2  -1 3 /
\arrow <8pt> [.2, .67] from   -1 4  to   -1 4.2
\arrow <8pt> [.2, .67] from    -0.5 3.5 to  -0.7 3.3
\arrow <8pt> [.2, .67] from    -0.5 4.5 to  -0.3 4.3
\put {$_{3}$}  at  -0.5 3.8
\put {$_{2}$}  at  -0.5 4.8
\put {$_{0}$}  at  -1.2 4
\putrule from  -1 3 to  -1 5
\setlinear \plot  -1 3  0 4  -1 5 /
\put {$\ell$}  at  -1 5.5
\put {$\ell_3$}  at  0 5.5
\setshadegrid span <1.5pt>
\vshade  -1  -1  3  <,z,,>   0   0 4  /   
\setcoordinatesystem units  <1.4cm, .8cm>  point at -7 0
\setplotarea  x from -1.2 to 0.1,  y from -1 to 5        
\arrow <8pt> [.2, .67] from   -1 0 to   -1 0.2
\arrow <8pt> [.2, .67] from    -0.5 -0.5 to -0.7 -0.7
\arrow <8pt> [.2, .67] from   -0.5 0.5 to -0.3 0.3
\arrow <8pt> [.2, .67] from    0 0.8 to  0 1
\arrow <8pt> [.2, .67] from    0 2.8 to  0 3
\putrule from -1 -1 to -1 1
\setlinear \plot -1 -1 0 0 -1 1 /
\put {$_{6}$}  at  -1.2 0
\put {$_{0}$} at  -0.4 -0.7
\put {$_{0}$}  at  -0.6 0.3
\put {$_{6}$}  at   0.2 1
\put {$_{6}$}  at   0.2 3
\arrow <8pt> [.2, .67] from   -1 2  to   -1 2.2
\arrow <8pt> [.2, .67] from    -0.5 1.5 to  -0.7 1.3
\arrow <8pt> [.2, .67] from    -0.5 2.5 to  -0.3 2.3
\put {$_{0}$}  at  -0.4 1.3
\put {$_{0}$}  at  -0.5 2.8
\put {$_{6}$}  at  -1.2 2
\putrule from  -1 1 to  -1 3
\putrule from 0 -1 to  0 5  
\setlinear \plot -1 1  0 2  -1 3 /
\arrow <8pt> [.2, .67] from   -1 4  to   -1 4.2
\arrow <8pt> [.2, .67] from    -0.5 3.5 to  -0.7 3.3
\arrow <8pt> [.2, .67] from    -0.5 4.5 to  -0.3 4.3
\put {$_{0}$}  at  -0.5 3.8
\put {$_{0}$}  at  -0.5 4.8
\put {$_{6}$}  at  -1.2 4
\putrule from  -1 3 to  -1 5
\setlinear \plot  -1 3  0 4  -1 5 /
\put {$\ell_3$}  at  -1 5.5
\put {$\ell_4$}  at  0 5.5
\setshadegrid span <1.5pt>
\vshade  -1  -1  1  <,z,,>   0   0 0  /
\endpicture
}
\caption{}
\end{figure}

Continuing to extend the tree, it turns out that there is only one further edge emanating from $\ell_3$ to a vertex $\ell_4$
as illustrated on the right of Figure \ref{3neighbours}.
As indicated by their edge labels, $\ell_3$, $\ell_4$ are in the same $\Gamma$-orbit, hence in the same $\Z_\Gamma(g)$-orbit, by Lemma \ref{Z-orbit}.
The element $x_3^{-1}x_0x_3$ acts with inversion on the
edge $[\ell_3, \ell_4]$ of $\cl T_g$, since it acts by glide reflection on the corresponding strip,
in which a fundamental domain is shaded. We therefore divide $[\ell_3, \ell_4]$ in two
by inserting a vertex $\ell_4^*$ at its midpoint, corresponding to the median line of the strip.
Since $x_3^{-1}x_0x_3$ acts by translation on this median line, it stabilizes $\ell_4^*$.
In fact $S(\ell_4^*)=\langle x_3^{-1}x_0x_3 \rangle\cong \bb Z$.

\refstepcounter{picture}
\begin{figure}[htbp]\label{55}
\hfil
\centerline{
\beginpicture
\setcoordinatesystem units  <1.4cm, .8cm>
\setplotarea  x from -1.2 to 0.1,  y from -1 to 5        
\arrow <8pt> [.2, .67] from   -1 0 to   -1 0.2
\arrow <8pt> [.2, .67] from    -0.5 -0.5 to -0.7 -0.7
\arrow <8pt> [.2, .67] from   -0.5 0.5 to -0.3 0.3
\arrow <8pt> [.2, .67] from    0 0.8 to  0 1
\arrow <8pt> [.2, .67] from    0 2.8 to  0 3
\putrule from -1 -1 to -1 1
\setlinear \plot -1 -1 0 0 -1 1 /
\put {$_{2}$}  at  -1.2 0
\put {$_{5}$} at  -0.4 -0.7
\put {$_{4}$}  at  -0.6 0.3
\put {$_{1}$}  at   0.2 1
\put {$_{1}$}  at   0.2 3
\arrow <8pt> [.2, .67] from   -1 2  to   -1 2.2
\arrow <8pt> [.2, .67] from    -0.5 1.5 to  -0.7 1.3
\arrow <8pt> [.2, .67] from    -0.5 2.5 to  -0.3 2.3
\put {$_{5}$}  at  -0.4 1.3
\put {$_{4}$}  at  -0.5 2.8
\put {$_{2}$}  at  -1.2 2
\putrule from  -1 1 to  -1 3
\putrule from 0 -1 to  0 5  
\setlinear \plot -1 1  0 2  -1 3 /
\arrow <8pt> [.2, .67] from   -1 4  to   -1 4.2
\arrow <8pt> [.2, .67] from    -0.5 3.5 to  -0.7 3.3
\arrow <8pt> [.2, .67] from    -0.5 4.5 to  -0.3 4.3
\put {$_{5}$}  at  -0.5 3.8
\put {$_{4}$}  at  -0.5 4.8
\put {$_{2}$}  at  -1.2 4
\putrule from  -1 3 to  -1 5
\setlinear \plot  -1 3  0 4  -1 5 /
\put {$\ell_1$}  at  -1 5.5
\put {$\ell_5$}  at  0 5.5
\setshadegrid span <1.5pt>
\vshade  -1  -1  1  <,z,,>   0   0 2  /  
\setcoordinatesystem units  <1.4cm, .8cm>  point at -2 0
\setplotarea  x from -1.2 to 0.1,  y from -1 to 5        
\arrow <8pt> [.2, .67] from   -1 0 to   -1 0.2
\arrow <8pt> [.2, .67] from    -0.5 -0.5 to -0.7 -0.7
\arrow <8pt> [.2, .67] from   -0.5 0.5 to -0.3 0.3
\arrow <8pt> [.2, .67] from    0 0.8 to  0 1
\arrow <8pt> [.2, .67] from    0 2.8 to  0 3
\putrule from -1 -1 to -1 1
\setlinear \plot -1 -1 0 0 -1 1 /
\put {$_{3}$}  at  -1.2 0
\put {$_{1}$} at  -0.4 -0.7
\put {$_{5}$}  at  -0.6 0.3
\put {$_{4}$}  at   0.2 1
\put {$_{4}$}  at   0.2 3
\arrow <8pt> [.2, .67] from   -1 2  to   -1 2.2
\arrow <8pt> [.2, .67] from    -0.5 1.5 to  -0.7 1.3
\arrow <8pt> [.2, .67] from    -0.5 2.5 to  -0.3 2.3
\put {$_{1}$}  at  -0.4 1.3
\put {$_{5}$}  at  -0.5 2.8
\put {$_{3}$}  at  -1.2 2
\putrule from  -1 1 to  -1 3
\putrule from 0 -1 to  0 5  
\setlinear \plot -1 1  0 2  -1 3 /
\arrow <8pt> [.2, .67] from   -1 4  to   -1 4.2
\arrow <8pt> [.2, .67] from    -0.5 3.5 to  -0.7 3.3
\arrow <8pt> [.2, .67] from    -0.5 4.5 to  -0.3 4.3
\put {$_{1}$}  at  -0.5 3.8
\put {$_{5}$}  at  -0.5 4.8
\put {$_{3}$}  at  -1.2 4
\putrule from  -1 3 to  -1 5
\setlinear \plot  -1 3  0 4  -1 5 /
\put {$\ell_2$}  at  -1 5.5
\put {$\ell_6$}  at  0 5.5
\setshadegrid span <1.5pt>
\vshade  -1  -1  1  <,z,,>   0   0 2  /  
\setcoordinatesystem units  <1.4cm, .8cm>  point at -4 0
\setplotarea  x from -1.2 to 0.1,  y from -1 to 5        
\arrow <8pt> [.2, .67] from   -1 0 to   -1 0.2
\arrow <8pt> [.2, .67] from    -0.5 -0.5 to -0.7 -0.7
\arrow <8pt> [.2, .67] from   -0.5 0.5 to -0.3 0.3
\arrow <8pt> [.2, .67] from    0 0.8 to  0 1
\arrow <8pt> [.2, .67] from    0 2.8 to  0 3
\putrule from -1 -1 to -1 1
\setlinear \plot -1 -1 0 0 -1 1 /
\put {$_{1}$}  at  -1.2 0
\put {$_{6}$} at  -0.4 -0.7
\put {$_{2}$}  at  -0.6 0.3
\put {$_{4}$}  at   0.2 1
\put {$_{4}$}  at   0.2 3
\arrow <8pt> [.2, .67] from   -1 2  to   -1 2.2
\arrow <8pt> [.2, .67] from    -0.5 1.5 to  -0.7 1.3
\arrow <8pt> [.2, .67] from    -0.5 2.5 to  -0.3 2.3
\put {$_{5}$}  at  -0.4 1.3
\put {$_{3}$}  at  -0.5 2.8
\put {$_{1}$}  at  -1.2 2
\putrule from  -1 1 to  -1 3
\putrule from 0 -1 to  0 5  
\setlinear \plot -1 1  0 2  -1 3 /
\arrow <8pt> [.2, .67] from   -1 4  to   -1 4.2
\arrow <8pt> [.2, .67] from    -0.5 3.5 to  -0.7 3.3
\arrow <8pt> [.2, .67] from    -0.5 4.5 to  -0.3 4.3
\put {$_{6}$}  at  -0.5 3.8
\put {$_{2}$}  at  -0.5 4.8
\put {$_{1}$}  at  -1.2 4
\putrule from  -1 3 to  -1 5
\setlinear \plot  -1 3  0 4  -1 5 /
\put {$\ell_5$}  at  -1 5.5
\put {$\ell_7$}  at  0 5.5
\setshadegrid span <1.5pt>
\vshade  -1  -1  3  <,z,,>   0   0 4  /   
\endpicture
}
\caption{}
\end{figure}

Continue to build the tree $\cl T_g^*$, adding vertices $\ell_5$, $\ell_6$, $\ell_7$ (as in Figure \ref{55}), noting
 that $\ell_6$, $\ell_7$ are in the same $\Z_\Gamma(g)$-orbit, by Lemma \ref{Z-orbit}.
The action of $G=\Z_\Gamma(g)/\langle g \rangle$ on $\cl T_g^*$ is described by the quotient graph of groups in Figure \ref{GG}.
The vertices of the quotient graph $G\backslash \cl T_g^*$ are labelled according to their inverse images in $\cl T_g^*$.
For example the stabilizer of the vertex $\ell_1$ is generated by an element $h$ conjugate to $x_2$: in fact $h=x_6^{-1}x_2x_6$.
The same applies to all vertices in the same $\Gamma$ orbit as $\ell_1$; that is all axes of $g$ with the periodic edge labelling $(\dots 2,2,2 \dots)$.
Thus the vertex $G.\ell_1$ of $G\backslash \cl T_g^*$ has label $(2)$.
Similarly, the label $[0]$ is attached to the vertex $G.\ell_4^*$. A square bracket is used to indicate
that $\ell_4^*$ is a median line rather than a wall.
The vertex and edge groups are shown in Figure \ref{GG}, except where they are trivial.
For example, attached to the vertex $(2)$ is the group $S(\ell_1)/\langle g \rangle=\langle x_6^{-1}x_2x_6 \rangle/\langle g \rangle\cong \bb Z/2\bb Z$.

Computing the fundamental group of this graph of groups as described in Section \ref{Bass-Serre}, using Lemma \ref{edgemap}, we see that
$$\Z_\Gamma(g)/\langle g \rangle\cong \bb Z * (\bb Z/2\bb Z)^{*2} * (\bb Z/4\bb Z).$$

\refstepcounter{picture}
\begin{figure}[htbp]
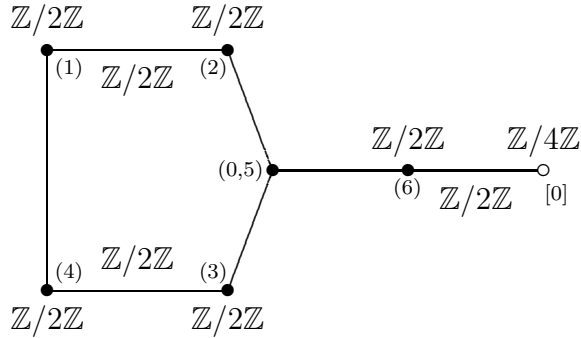
\label{GG}
\centerline{
\beginpicture
\setcoordinatesystem units <0.6cm,0.8cm>   
\setplotarea x from -3 to 9, y from -2.5  to 2.5         
\putrule from -2 -2 to 2 -2
\putrule from 2 2 to -2 2
\putrule from -2 2 to -2 -2
\setlinear \plot  2 2  3 0  2 -2 /
\putrule from  3 0  to  8.865 0
\put{$\bullet$}  at  -2 -2
\put{$\bullet$}  at  2 2
\put{$\bullet$}  at  2 -2
\put{$\bullet$}  at  -2 2
\put{$\bullet$} at   3 0
\put{$\bullet$}  at 6 0
\put{$\circ$}  at 9 0
\put{$_{(4)}$}  at  -1.5 -1.7
\put{$_{(2)}$}  at  1.7 1.7
\put{$_{(3)}$}  at  1.7 -1.7
\put{$_{(1)}$}  at  -1.5 1.7
\put{$_{(0,5)}$} at   2.3 0
\put{$_{(6)}$}  at 6 -0.3
\put{$_{[0]}$}  at 9.3 -0.4
\put{$\bb Z/2\bb Z$}  at  -2 -2.5
\put{$\bb Z/2\bb Z$}  at  2 2.5
\put{$\bb Z/2\bb Z$}  at  2 -2.5
\put{$\bb Z/2\bb Z$}  at  -2 2.5
\put{$\bb Z/2\bb Z$}  at 6 0.5
\put{$\bb Z/4\bb Z$}  at 9 0.5
\put{$\bb Z/2\bb Z$}  at  0 1.5
\put{$\bb Z/2\bb Z$}  at  0 -1.5
\put{$\bb Z/2\bb Z$}  at  7.5 -0.5
\endpicture
}
\caption{Graph of groups with fundamental group $\bb Z * (\bb Z/2\bb Z)^{*2} * (\bb Z/4\bb Z)$}
\end{figure}

\subsection{Example: the graph of groups for $g=x_0x_1x_4$}
In this case the corresponding graph of groups is illustrated in Figure \ref{GG2}.
This may be verified as before, and the notation used for the vertices is the same.
For example, the vertex labelled $[2,3,5]$ is the image of a vertex of $\cl T_g^*$
 which is stabilized by an element conjugate to $x_2x_3x_5$. Computing the fundamental group of this graph of groups gives
$$\Z_\Gamma(g)/\langle g \rangle\cong \bb Z^{*2} * (\bb Z/2\bb Z)^{*5}.$$

\refstepcounter{picture}
\begin{figure}[htbp]
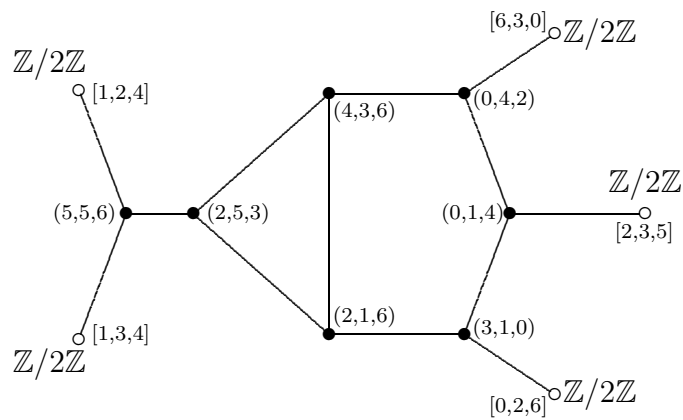
\label{GG2}
\centerline{
\beginpicture
\setcoordinatesystem units <0.6cm,0.8cm>   
\setplotarea x from -4 to 4, y from -2.5  to 3.5         
\putrule from -1 -2 to 2 -2
\putrule from 2 2 to -1 2
\putrule from -1 2 to -1 -2
\setlinear \plot  2 2  3 0  2 -2 /
\setlinear \plot  -1 2  -4 0  -1 -2 /
\putrule from  3 0  to  5.865 0
\putrule from  -4 0  to  -5.5 0
\setlinear \plot  -6.5 2   -5.5 0  -6.5 -2 /
\setlinear \plot 2 2    3.9 2.95  /
\setlinear \plot  2 -2    3.9 -2.95  /
\put{$\bullet$}  at  -1 -2
\put{$\bullet$}  at  2 2
\put{$\bullet$}  at  2 -2
\put{$\bullet$}  at  -1 2
\put{$\bullet$} at   3 0
\put{$\bullet$} at   -4 0
\put{$\bullet$} at   -5.5 0
\put{$\circ$}  at 6 0
\put{$\circ$}  at 4 3
\put{$\circ$}  at 4 -3
\put{$\circ$}  at -6.55 2.05
\put{$\circ$}  at -6.55 -2.09
\put{$_{(2,1,6)}$}  at  -0.2 -1.7
\put{$_{(0,4,2)}$}  at  2.9 1.9
\put{$_{(3,1,0)}$}  at  2.9 -1.9
\put{$_{(4,3,6)}$}  at  -0.2 1.7
\put{$_{(0,1,4)}$} at   2.2 0
\put{$_{(2,5,3)}$} at   -3 0
\put{$_{(5,5,6)}$} at   -6.4 0
\put{$_{[1,2,4]}$} at   -5.6 2
\put{$_{[1,3,4]}$} at   -5.6 -2
\put{$_{[2,3,5]}$} at   6 -0.3
\put{$_{[6,3,0]}$} at   3.2 3.2
\put{$_{[0,2,6]}$} at   3.2 -3.2
\put{$\bb Z/2\bb Z$}  at  -7.2 2.5
\put{$\bb Z/2\bb Z$}  at  -7.2 -2.5
\put{$\bb Z/2\bb Z$}  at 6 0.5
\put{$\bb Z/2\bb Z$}  at 5 3
\put{$\bb Z/2\bb Z$}  at 5 -3
\endpicture
}
\caption{Graph of groups with fundamental group $\bb Z^{*2} * (\bb Z/2\bb Z)^{*5}$}
\end{figure}

   Other groups for which such computations are possible are those which act regularly on the vertices of a building of type $\widetilde A_1 \times \widetilde A_1$, that is a product of trees. Explicit presentations of such groups are given in \cite{kr}. The groups studied in \cite{rr} have the property that the centralizer of any non-trivial element is either $\bb Z$ or $\bb Z^2$ (with both possibilities occurring).

\end{document}